\documentclass{amsart}
\usepackage[english]{babel}
\usepackage{amssymb}

\newcommand{\cF}{{\mathcal{F}}}
\newcommand{\cE}{{\mathcal{E}}}

\newcommand{\cO}{{\mathcal{O}}}

\newcommand{\cL}{{\mathcal{L}}}
\newcommand{\cP}{{\mathcal{P}}}

\newcommand{\Ps}{{\mathbf{P}}}
\newcommand{\Z}{{\mathbf{Z}}}
\newcommand{\C}{{\mathbf{C}}}

\newcommand{\Q}{{\mathbf{Q}}}
\newcommand{\A}{{\mathbf{A}}}
\newcommand{\F}{{\mathbf{F}}}

\newcommand{\cQ}{{\mathcal{Q}}}

\newcommand{\str}{{\mathcal{O}}}

\renewcommand{\phi}{\varphi}
    \newtheorem{lemma}{Lemma}[section]
    \newtheorem{proposition}[lemma]{Proposition}
    \newtheorem{theorem}[lemma]{Theorem}
    
    \newtheorem{corollary}[lemma]{Corollary}

   \theoremstyle{definition}
    \newtheorem{definition}[lemma]{Definition}

    \DeclareMathOperator{\rank}{rank}
    \DeclareMathOperator{\Pic}{{Pic}}

\DeclareMathOperator{\sing}{{sing}}

\DeclareMathOperator{\MW}{{MW}}

\DeclareMathOperator{\sm}{{sm}}

\DeclareMathOperator{\Div}{{Div}}

\usepackage[all]{xy}
\begin{document}
\title[Rank of elliptic $n$-folds]{The average rank of  elliptic $n$-folds}
\author{Remke Kloosterman}
\address{Institut f\"ur Mathematik, Humboldt-Universit\"at zu Berlin, Unter den Linden 6, D-10099 Berlin, Germany}
\email{klooster@math.hu-berlin.de}

\begin{abstract} Let $V/\F_q$ be a variety of dimension at least two. We show that the density of  elliptic curves $E/\F_q(V)$ with positive rank is zero if $V$ has dimension at least 3 and is at most $1-\zeta_V(3)^{-1}$ if $V$ is a surface.
\end{abstract}
\date{\today}
\thanks{The author would like to thank Torsten Ekedahl for pointing out a proof for Proposition~\ref{prpMV},
 Miles Reid for providing the reference \cite[Th\'eor\`eme XI.3.13]{SGA2} and Matthias Sch\"utt and Orsola Tommasi for giving several comments on a previous version of this paper. The author would like to thank the referee for suggesting improvements in the presentation.}
\maketitle

\section{Introduction}\label{secInt}
Let $N$ be a positive integer, consider the set $\mathcal{EC}_N$ consisting of elliptic curves $y^2=x^3+Ax+B$ such that $A,B\in \Z$, the Weierstrass equation is minimal (i.e., there is no $u\neq \pm 1$ such that $u^4$ divides $A$ and $u^6$ divides $B$) and such that the absolute value of the discriminant $| 4A^3+27B^2 |$ is at most $N$. Consider the ``average rank''
\[\lim_{N \to \infty} \frac{\sum_{ E\in \mathcal{EC}_N} \rank E(\Q)  }{\# \mathcal{EC}_N}.\]
Not very much is known about this limit, even its existence is not clear. Recently, Bhargava and Shankar \cite{BhaShaTer} announced that the limsup of this sequence is at most $1.17$ and that the liminf is nonzero.
It is conjectured that the average rank is  $1/2$. Moreover, there is a collection of refined conjectures. One of these is that for $\epsilon\in \{0,1\}$ we have
\[ \mu_{\epsilon}:=\lim_{N \to \infty} \frac{\#\{ E\in \mathcal{EC}_N \mid  \rank E(\Q)= \epsilon  \}}{\# \mathcal{EC}_N}=\frac{1}{2}.\]
In particular, the elliptic curves with rank at least 2 have density zero among all elliptic curves.

We switch to the case of elliptic curves over function fields of curves. I.e., we fix a geometrically irreducible curve $C/\F_q$ and consider the set of isomorphism classes of elliptic curves $E/\F_q(C)$. We order these curves by the degree of the minimal discriminant. Since for a given degree there are only finitely many isomorphism classes of elliptic curves with minimal discriminant of that degree, we can define $\mu_\epsilon$ for elliptic curves over $\F_q(C)$ and we can define the average rank of elliptic curves over $\F_q(C)$.
Again one expects  $\mu_0=\mu_1=1/2$.

In the number field cases these conjectures are justified by a combination of the Birch--Swinnerton-Dyer conjecture and several conjectures on (the $L$-series of) modular forms. In the function field case one exploits that  an elliptic curve $E/\F_q(C)$ has an associated  surface $S/\F_q$.  The characteristic polynomial of Frobenius on the second cohomology of $S$ behaves ``analogously'' to the $L$-series of a modular form and one uses this to justify the conjectures in this case.
For an extended discussion of this see e.g. \cite[Introduction]{KatzLfunc}, for a different point of view see \cite{dJonES}.

The purpose of this paper is to show that analogous conjectures do not hold if one considers elliptic curves over $\F_q(V)$, where $V$ is a variety of dimension at least 2.

To ease the presentation we discuss in the introduction only the case where the characteristic of $\F_q$ is at least 5.  The main results, however, hold also in characteristic 2 and 3.

To formulate our main result we need to fix some more notation: Let $V/\F_q$ be  a smooth geometrically irreducible projective variety of dimension at least 2. Let $\cL$ be an ample line bundle on $V$. 
Then an element of $f\in K(V)$ can be written as $f=s/t$ with $s,t\in H^0(\cL^k), t\neq 0$ for some $k>0$. The degree of  $f$ is the minimal possible such $k$ in a representation $f=s/t$.
After replacing $\cL$ by some tensorpower, if necessary, we may assume that $h^0(\cL)\neq 0$. Fix a nonzero element $t\in H^0(\cL)$, let $R:=\oplus_{k\geq 0} H^0(\cL^k)$.
Then for  each elliptic curve $E/\F_q(V)$ we can find a Weierstrass equation $y^2=x^3+A/t^{4k}x+B/t^{6k}$ such that $A\in H^0(\cL^{4k}), B\in H^0(\cL^{6k})$. Moreover, we may assume that there is no $f\in R\setminus \F_q$ such that $f^4$ divides $A$ and $f^6$ divides $B$. Under this assumption the pair $(A,B)$ is uniquely determined by $E$ up to the action of $\F_q^*$ with weight $(4,6)$ on $(A,B)$. We call such pairs $(A,B)$ minimal.

For a scheme $X/\F_q$ define Weil's zeta function
\[ \zeta_X(s)=Z_X(q^{-s}):=\prod_{P\in X \mbox{closed}}(1- q^{-\deg(P)s})^{-1} =  \exp \left( \sum_{r=1}^\infty  \frac{\#X(\F_{q^r})}{r} q^{-rs} \right).\] 

The main result we prove is the following:
\begin{theorem}\label{mainThm} Let $R$, $t$  and $V/\F_q$ be as above. Let 
\[ \mathcal{E}_N=\left\{ E_{A,B} : y^2=x^3+At^{-4k}x+Bt^{-6k} \left| \begin{array}{l}A,B \in H^0(\cL^{4k})\times H^0(\cL^{6k})\\ \mbox{for some } k\leq N,\\ (A,B) \mbox{ is minimal} \\ \mbox {and }  4A^3+27B^2\neq 0\end{array}\right. \right\}. \]
Then 
\[ \mu_0:=\liminf_{N\to\infty} \frac{\# \{ E_{A,B} \in  \mathcal{E}_N \mid \rank E_{A,B}(\F_q(V))=0\}}{\# \mathcal{E}_N } \geq \zeta_{V}(\dim V+1)^{-1}.\]
Moreover, if $\dim V\geq 3$ then $ \mu_0=1$.
\end{theorem}
In the case $\dim V=2$ we have that $\mu_0>1/2$ if $\zeta_{V}(3)\leq 2$. 
This inequality is achieved  if we replace $V/\F_q$ by $V_{\F_{q^r}}/\F_{q^r}$ for $r$ sufficiently large.

The proof of both results consists of two parts. The first part is geometric. Each $E_{A,B} \in \mathcal{E}_k$ yields a hypersurface $Y_{A,B}$ in $\Ps(\str\oplus \cL^{-2k}\oplus \cL^{-3k})$.  We show that if the rank of $E_{A,B}$ is positive then $Y_{A,B}$ has a Weil divisor that is not $\Q$-Cartier. This is the point in the proof where we use $\dim V>1$. I.e., this claim is false if $\dim V=1$.
Since $Y_{A,B}$ has a Weil-divisor that is not $\Q$-Cartier and $Y_{A,B}$ has only hypersurface singularities, it follows from  \cite[Th\'eor\`eme XI.3.13]{SGA2} that  $Y_{\sing}$ has codimension 2 or 3.

We then use this geometric result to find a lower bound for $\mu_0$. This is done by applying a modification of Poonen's Bertini Theorem over finite fields \cite{Poo}. This (modified) result yields directly the densities mentioned in Theorem~\ref{mainThm}.

It seems unlikely that the bound for $\mu_0$ in the two-dimensional case is sharp. Using the ideas presented in \cite{ell3HK} one expects that if  $E/\C(s,t)$ has positive  Mordell--Weil rank then the corresponding hypersurface $Y$ in $\Ps(\str\oplus \cL(-2k) \oplus \cL(-3k))$ has Milnor number at least $12\deg(\cL)^2 k^2$. A similar result should hold in  positive characteristic, i.e., one needs a lot of singularities on $Y$ in order to have positive Mordell--Weil rank. In our density calculation we only use that $Y$ is singular. However, we did not manage to adapt the above observation in the density calculations.

The organization of this paper is as follows:
In Section~\ref{sectMod} we discuss the algebraic part of the proof, i.e., we prove that if $\rank E(\overline{\F_q}(V))$ is positive then the model in $\Ps(\str \oplus \cL^{-4k}\oplus \cL^{-6k})$ is singular in codimension 3. In Section~\ref{sectDen} we calculate the densities.

\section{Model in a $\Ps^2$-bundle}\label{sectMod}
Let $p$ be a prime number. Let $q=p^r$ be a prime power. Let $V/\F_q$ be a smooth and geometrically irreducible variety of dimension $n-1$, with $n\geq 3$. 

Let $\cL$ be a very ample line bundle on $V$ and $k>0$ be an integer. Define $\cE_k=\cO\oplus \cL^{-2k} \oplus \cL^{-3k}$.
Then $\Ps(\cE_k)$ is a $\Ps^2$-bundle over $V$. Let $\pi: \Ps(\cE_k) \to V$ be the bundle projection.
We use Grothendieck's definition of projective space, in particular,
$\pi_*\cO_{\Ps(\cE_k)}(1)=\cE_k$.
 Fix sections \begin{eqnarray*}
X:=(0,1,0)& \in & H^0(\cL^{2k} \oplus \cO \oplus \cL^{-k})=H^0(\str_{\Ps(\cE_k)}(1)
\otimes \cL^{2k}),\\
 Y:=(0,0,1) &\in& H^0(\cL^{3k} \oplus \cL^k \oplus \cO)=H^0(\str_{\Ps(\cE_k)}(1)
\otimes \cL^{3k}),\\ 
 Z:=(1,0,0) &\in& H^0(\cO \oplus \cL^{-2k}\oplus \cL^{-3k})=H^0(\str_{\Ps(\cE_k)}(1)
).\end{eqnarray*}

For $i=1,2,3,4,6$ fix sections $a_i$ in $H^0(\cL^{ik})$. Then
\begin{equation}\label{HSinPB} Y^2Z+a_1XYZ+a_3YZ^2=X^3+a_2X^2Z+a_4XZ^2+a_6Z^3\end{equation}
defines a  hypersurface $W$ in $\Ps(\cE)$. Let $\Delta \in H^0(\cL^{12k})$ be the discriminant of this equation. If $\Delta$ does not vanish on all of $V$ then the general fiber of the projection $\pi|_W:W\to V$ is an elliptic curve, where we take $X=Z=0$ as the zero-section.
In particular, we obtain an elliptic curve $E/\F_q(V)$.
Conversely, given an elliptic curve $E/\F_q(V)$ we can find an integer $k$ and sections $a_i$ such that the generic fiber of (\ref{HSinPB}) is isomorphic to $E/\F_q(V)$.
The $a_i$ are not unique. Without loss of generality we may assume that there exists no positive integer $j$ and no nonzero element $u\in H^0(\cL^j)$ such that for all $i$ we have that $u^i$ divides $a_i$ in the ring $\oplus_t H^0(\cL^t)$, i.e., we have a so-called minimal Weierstrass equation.

It is well-known that  without loss of generality we may assume the following:
\begin{itemize}
 \item If $p=2$ then $a_2=0$.
\item If $p=3$ then $a_1=a_3=0$.
\item If $p>3$ then $a_1=a_2=a_3=0$.
\end{itemize}
This assumption simplifies the calculation of densities later on.

We show now that if   $E(\overline{\F_q}(V))$ is infinite then $W$ has a Weil divisor that is not $\Q$-Cartier. We start by studying the Cartier divisors on $W$. To this end we consider the \'etale cohomology of $W_{\overline{\F_q}}$. For the rest of this section fix a prime number $\ell\neq p$. All cohomology groups under consideration are \'etale cohomology groups of the appropiate variety over ${\overline{\F_q}}$.

\begin{definition}
 Let $f: X\to Y$ be a proper morphism between quasi-projective varieties. We say that $f$ is a \emph{proper modification} if there is a proper closed subset $Z\subsetneq Y$ such that $f|_{X\setminus f^{-1}(Z)}: X\setminus f^{-1}(Z) \to Y\setminus Z$ is an isomorphism. The \emph{discriminant of $f$} is the minimal closed subset $\Delta \subset Y$ such that $f|_{X\setminus f^{-1}(\Delta)}: X\setminus f^{-1}(\Delta) \to Y\setminus \Delta$ is an isomorphism. We call $E=f^{-1}(\Delta)$ the \emph{exceptional locus of $f$}.
\end{definition}

\begin{proposition}[Gysin sequence]\label{prpGysin} Let $X$ be a projective variety, $Y\subset X$ be a closed subvariety. Then we have an exact sequence of \'etale cohomology groups:
 \[ H^i_c(X\setminus Y,\Z_\ell)\to H^i_c(X,\Z_\ell) \to H^i_c(Y,\Z_\ell) \to H_c^{i+1}(X\setminus Y,\Z_\ell) \dots \]
\end{proposition}
\begin{proof}
Let $\iota: Y\to X$ and $j:X\setminus Y \to X$ be the inclusion then for each constructable sheaf $\cF$ on $X$ we have the exact sequence
\[ 0 \to j_!j^*\cF \to \cF \to \iota_*\iota^* \cF\to 0.\]
Take now $\cF=\Z_\ell$. Then the cohomology of the first sheaf is by definition the cohomology with compact support of $X\setminus Y$. Since $X$ and $Y$ are proper we obtain the following long exact sequence 
\[H^i_c(X\setminus Y,\Z_\ell) \to H^i_c(X,\Z_\ell) \to H^i_c(Y,\Z_\ell) \to \dots\]
\end{proof}

The existence of the following exact sequence seems to be well-known, but we could not find any reference in the case of \'etale cohomology.

\begin{proposition}\label{prpMV} Let $f:X\to Y$ be a proper birational map between quasi-projective varieties. Let $\Delta \subset Y$ and $E \subset X$ be closed subsets, such that $f$ maps $X\setminus E $ isomorphically to $Y\setminus \Delta$. Then the following sequence is exact
\[ H^i(Y)\to H^i(X)\oplus H^i(\Delta) \to H^i(E) \to H^{i+1}(Y) \to \dots
\]
\end{proposition}
\begin{proof}
Let  $U:=X\setminus E$ and $V:=Y\setminus \Delta$ and let $\iota:U \to X$ and $j:V\to Y$ be the inclusions. We have the following commutative diagram
\[ \xymatrix{ 
\dots\ar[r]&H^i(Y,j_!\Z_{\ell}) \ar[r]\ar[d] & H^i(Y,\Z_{\ell})\ar[r]\ar[d] & H^i(\Delta,\Z_{\ell}) \ar[r]\ar[d]&\dots\\
\dots\ar[r]&H^i(X,\iota_!\Z_{\ell}) \ar[r] & H^i(X,\Z_{\ell})\ar[r] & H^i(E,\Z_{\ell}) \ar[r]&\dots\\
}\]
where the vertical arrows are induced by $f$. We prove now that $f^*:H^i(Y,j_!\Z_{\ell})\to H^i(X,\iota_!\Z_{\ell})$ is an isomorphism. 
Since $f$ is proper we have by definition that $R^k f_*\iota_! \Z_\ell=R^k(f \iota) _!\Z_\ell$. 
Let $f'=f|_U$ then we have that $(f\iota)_!=(jf')_!=j_!f'_!$. Since $f'$ is an isomorphism, we have $f'_!\Z_\ell=\Z_\ell$. This yields $R^k (f\iota)_! \Z_\ell=R^kj_!f'_!\Z_\ell=R^kj_!\Z_\ell$. The right hand side clearly vanishes for $k>0$.
 Using the Leray spectral sequence we obtain the desired isomorphism
\[H^i(X,\iota_! \Z_{\ell})\cong \oplus_{k+m=i} H^m(Y,R^k f_* \iota_!\Z_{\ell})=H^i(Y,j_! \Z_{\ell}).\]
A standard diagram chase finishes the proof.
\end{proof}

We return to our hypersurface $W$.

\begin{proposition}\label{prpLef} The map $\pi|_W^*: H^2(V) \to H^2(W)$ is injective and has a one-dimensional cokernel. Let $\iota:W\to W$ be the fiberwise elliptic involution. Then $\iota^*:H^2(W)\to H^2(W)$ is the identity map. Similarly, we have that $\pi|_W^*: H^1(V) \to H^1(W)$ is an isomorphism.
\end{proposition}

\begin{proof}
Since $\pi|_W\circ \sigma$ is the identity, it follows that $\sigma^*\circ \pi|_W^*$ is an isomorphism, hence $\pi|_W^*$ is injective.

Use the linear system $|\cL|$ to embed $V$ into some $\Ps^m$. Consider the weighted projective space $\Ps(2t,3t,1^{m+1})$, where $t=k\deg(\cL)$. Let $x,y,z_0,\dots,z_m$ be coordinates on $\Ps(2t,3t,1^{m+1})$.
Consider the projection map $\psi: \Ps(2t,3t,1^{m+1})\dashrightarrow \Ps^m$. Let $\Ps_0:=\overline{\psi^{-1}(V)}$. Let $W_0\subset \Ps_0$ be the hypersurface defined by
\[ y^2+a_1xy+a_3y=x^3+a_2x^2+a_4x+a_6.\]
There is a natural map $f: \Ps(\cE) \to \Ps_0$. This map is a proper modification, i.e., the divisor $\{Z=0\}$ is mapped to the rational curve $\{z_0=z_1=\dots=z_m=0\}$, and $f$ is an isomorphism on the complements.
The restriction $f|_W$ maps the image $\Sigma$ of the zero-section to the point $p=(1:1:0:0:\dots:0)$ and yields an isomorphism $W\setminus \Sigma \to W_0\setminus \{p\}$. 

Now $\Ps_0\setminus W_0$ is the intersection of $\Ps_0$ with the complement of a hypersurface in $\Ps(2t,3t,1^{m+1})$, hence $U:=\Ps_0\setminus W_0$ is affine. Since $n+1:=\dim U>3$ it follows that $H^{2n-i}(U)=0$, for $i\leq 1$, and therefore we have that $H^{2n}(W_0)\cong H^{2n}(\Ps_0)$.

The map $\psi:\Ps(\cE)\setminus W \to U$ is  a proper modification.  The center of this modification is a curve isomorphic to $\A^1$ and the exceptional divisor is a line bundle over $V$, hence the cohomology of the exceptional divisor $E$ is isomorphic to the cohomology of $V$.

From Proposition~\ref{prpMV} we obtain exact sequences for $i=-1,0,1$
\[ 0=H^{2n-i}(U)\to H^{2n-i}(\Ps(\cE)\setminus W) \to H^{2n-i}(V) =0,\]
where we use that $\dim V=\dim U-2=n-1$.

Hence $H^{2n-i}(\Ps(\cE)\setminus W)=0$. Since $\Ps(\cE)\setminus W$ is smooth we obtain by Poincar\'e duality that $H^i_c(\Ps(\cE)\setminus W)=0$ for $i=1,2,3$. From the Gysin exact sequence (Proposition~\ref{prpGysin}) it follows that $H^i_c(\Ps(\cE))\cong H^i_c(W)$  for $i=1,2$.

The cohomology of projective bundles is well-known, in particular, we have $H^1_c(\Ps(\cE)) \cong \pi^*H^1_c(V)$. This yields the statement concerning $H^1$. 

For the statement concerning $H^2$ note that $H^2_c(\Ps(\cE)) \cong \pi^*H^2_c(V) \oplus \Q_\ell$, where the second summand can be generated by the first chern class of the divisor $\{Z=0\}$. This implies that $H^2_c(W)\cong \pi|_W^*H^2_c(V) \oplus \Q_\ell$, where the second summand is the first chern class of the zero-section $\{X=Z=0\}$.

On $\Ps(\cE)$ one has a fiberwise involution $\iota:[X:Y:Z]\mapsto[X:-Y:Z]$. Restricted to $W$ this is just the elliptic involution on the fiber. Now $\iota^*$ leaves $H^2_c(V)$ invariant and fixes the zero-section, hence it leaves $H^2_c(W)=H^2(W)$ invariant.
\end{proof}

We want to define a cycle class map for cycles on $W$. The usual definition of the cycle class map  for cycles on a variety $X$ (cf. \cite[Section 23]{LEC}) assumes that $X$ is smooth. We will give a different, but equivalent definition of the cycle class map that works also on singular varieties and mimics the cycle class map of Borel-Moore homology.

For a variety $X$ denote with $C^r(X)$ the free abelian group generated by irreducible codimenion $r$ subvarieties, denote with $CH^r(X)$ the quotient of $C^r(X)$ modulo rational equivalence.

Suppose for the moment that $X$ is a smooth projective variety, and $Z$ is an irreducible subvariety of codimension $r$. There is a natural isomorphism $H^0(Z_{\sm},\Q_\ell)\cong \Q_\ell$.
One  defines $c_r(Z)$ the class of $Z$ as the image of $1$ under the composition
\[ \Q_\ell\to H^0(Z_{\sm},\Q_\ell)\to H^{2r}_{Z_{\sm}}(X\setminus Z_{\sing},\Q_\ell)=H^{2r}_Z(X,\Q_\ell)\to H^{2r}(X,\Q_\ell).\]
In order to generalize to the case where $X$ might be singular, we modify this map by composing it with Poincar\'e duality. Extending the composed map linearly map yields a linear map
\[ c_r^*: C^r(X) \to H^{2n-2r}_c(X,\Q_\ell)^*.\]
There is an alternative definition of this map, using that Poincar\'e duality is functorial. For an irreducible subvariety $Z$ of $X$ consider the composition
\[ H^0(Z_{\sm},\Q_\ell)\to H^{2n-2r}_c(Z_{\sm},\Q_\ell)^* \to H^{2n-2r}_c(Z,\Q_\ell)^*\to H^{2n-2r}_c(X,\Q_\ell)^*.\]
The first map is Poincar\'e duality. Since $\dim Z_{\sing}\leq n-r-1$ it follows from the Gysin sequence (\ref{prpGysin}) that $H^{2n-2r}_c(Z_{\sm},\Q_\ell) \cong H^{2n-2r}_c(Z,\Q_\ell)$ this defines the second map, the third map is the dual of the map from Proposition~\ref{prpGysin}.
In this alternative definition of $c_r^*$ we did not use that $X$ was smooth. 

Let us return to our hypersurface $W$. For a prime divisor $D$ on $W$ define the class of $D$ as the image of $1$ under
\[ H^0(D_{\sm},\Q_\ell)\to H^{2n-2}_c(D_{\sm},\Q_\ell)^* \to H^{2n-2}_c(D,\Q_\ell)^*\to H^{2n-2}_c(W,\Q_\ell)^*.\]
Extending this map linearly  defines a linear map
\[ c_1^*: \Div(W) \to H^{2n-2}_c(W,\Q_\ell)^*\]
Since $W$ is smooth in codimension 1, the Gysin sequence for $(W,W_{\sing})$ yields an isomorphism $H^{2n-2}_c(W_{\sm},\Q_\ell)\to H^{2n-2}_c(W,\Q_\ell)$. In particular, $c_1^*$ of a divisor of function vanishes.

A point $P\in E(\overline{\F_q}(V))$ can be considered as a rational section $\sigma_P:V\to W$ of $\pi|_W$. Denote with $\Sigma_P$ the closure of the image of $\sigma_p$. Let $W_\eta$ be the generic fiber of $\pi:W\to V$. There is a natural map $E(\overline{\F_q}(V))\to CH^1(W_\eta)$ by sending $P$ to $\Sigma_P-\Sigma_O$. From \cite{Waz} it follows that there is a natural section $CH^1(W_\eta)\to CH^1(W)$ to the restriction map $CH^1(W_\eta)\to CH^1(W)$. This defines a map $E(\overline{\F_q}(V))\to CH^1(W)$  sending $P \to \Sigma_P-\Sigma_O+D_P$, where $D_P$ is some divisor such that $\overline{\pi(D_P)}\neq V$. (For the details how to find $D_P$ see \cite{Waz}.)
Since all fibers are irreducible, we have that $D_P$ is the pullback of a divisor from $V$.

Let $H^{2n-2}(W,\Q_{\ell})^{*-}$ be the $-1$-eigenspace of $\iota^*$.
There is a natural map from $H^{2n-2}_c(W,\Q_{\ell})^* \to H^{2n-2}_c(W,\Q_{\ell})^{*-}$  mapping $a$ to $a-\iota^*(a)$. (Note that the class of $D_P$ is mapped to zero under this map.) Composing $c_1^*$ with this map, we obtain a natural homomorphism of groups $E(\overline{\F_q}(V)) \to H^{2n-2}_c(W,\Q_{\ell})^{*-}$.

\begin{lemma}  The natural map 
\[ E(\overline{\F_q}(V)) \otimes \Q_\ell  \to H^{2n-2}(W,\Q_\ell)^{-} \]
is injective.
\end{lemma}

\begin{proof} 
Note that we work over $\overline{\F_q}$. In particular,  we can find a curve $C$ that satisfies any  property that holds for a  general $C\subset V$.
Let $C$ be a curve in $V$. If $C$ is chosen sufficiently general, then  the general fiber of $\pi$ restricted to $W_C:=\pi^{-1}(C)$ is an elliptic curve. Then $W_C\to C$ defines an elliptic curve $E_C/\F_q(C)$.
Since $W$ is smooth in codimension 1 we have a natural isomorphism $H^{2n-2}_c(W_{\sm})^*\to H^{2n-2}_c(W)^*$. Since the general fiber of $W_C\to C$ is smooth, we have that $W_C$ has only isolated singularities, and hence $H^{2}_c((W_C)_{\sm})^*\cong H^2_c(W_C)^*$. Composing these isomorphisms with
\[  H^{2n-2}_c(W_{\sm})^*\stackrel{pd}{\to} H^2(W_{\sm})\to H^2((W_C)_{\sm})\stackrel{pd}{\to} H^2_c((W_C)_{\sm})^*\]
yields a map $ H_c^{2n-2}(W)^* \to H^2_c(W_C)^*$,
which commutes with $c_1^*$. Hence it makes sense to  consider the following diagram
\[ \xymatrix{ E(\overline{\F_q}(V))\otimes \Q \ar[r]\ar[d] & H^{2n-2}_c(W,\Q_{\ell})^* \ar[d] \\
E_C(\overline{\F_q}(C))\otimes \Q \ar[r] & H^2_c(W_C,\Q_{\ell}). }
\]
If $C$ is chosen sufficiently general then the specialization map \[E(\overline{\F_q}(V)) \to E_C(\overline{\F_q}(C))\] is injective. Since $W$ is minimal it follows that for a general $C$ the surface $W_C$ has at most isolated $ADE$ singularities. 

Let $\varphi: \tilde{W_C}\to W_C$ be a minimal resolution of singularities. It is well-known that $\tilde{W_C}\to W_C \to C$ is an elliptic surface. The exceptional divisors of $\varphi$ are precisely the fiber components not intersecting the zero-section.
Shioda \cite{ShiMW} defined a  pairing on the Mordell-Weil group of $\tilde{\pi}|_C:\tilde{W_C}\to C$. This pairing is induced from the intersection pairing, but takes values in $\Z[1/N]$ for some integer $N$. However, if we restrict the pairing to  the subgroup of sections intersecting the same fiber-component of each fiber as the identity component then the pairing takes values in $\Z$. This latter subgroup has finite index in $\MW(\tilde{\pi|_C})$, and is called the narrow Mordell-Weil lattice.
The narrow Mordell-Weil lattice  injects into $H^2(\tilde{W}_C,\Q_{\ell})$ and the image is orthogonal to the fiber components not intersecting the zero-section (cf.  \cite{ShiMW}). From this it follows that the image of the narrow Mordell-Weil lattice is contained in $\varphi^*(H^2(W_C,\Q_{\ell}))$. Since the image of the narrow Mordell-Weil lattice is orthogonal to the zero-section and the class of a general fiber, it follows that the narrow Mordell-Weil lattice is contained in 
$\varphi^*(H^2(W_C,\Q_{\ell})^{-})$. Hence
\[ E_C(\overline{\F_q}(C)) \otimes \Q \hookrightarrow H^2(W_C,\Q_{\ell}) ^{-}
\]
is injective and therefore
\[ E(\overline{\F_q}(V))\otimes \Q_\ell \to H^{2n-2}(W,\Q_\ell) \]
is injective. Since $H^{2n-2}(W,\Q_\ell) \to H^2(W_C,\Q_\ell)$ is $\iota^*$-equivariant it follows that
\[ E(\overline{\F_q}(V)) \otimes \Q_\ell \to H^{2n-2}(W,\Q_\ell) \to H^{2n-2}(W,\Q_\ell)^{-} \]
is injective.
\end{proof}

With $\Pic(W)$ we denote the group of Cartier divisors on $W$ modulo linear equivalence. There is a natural map from $\Pic(W)\to CH^1(W)$ the group of codimension 1 cycles in $W$ modulo rational equivalence. 
\begin{proposition}
Let $N(W)\subset CH^1(W)$ be the subgroup generated by the divisors on $W$ that are homologically trivial. Let $P(W)\subset CH^1(W)$ be the subgroup generated by $N(W)$ and $\Pic(W)$. Then there is an injection
\[ E(\overline{\F_q}(V)) \otimes \Q \to  (CH^1(W)/P(W))\otimes \Q. \]
\end{proposition}

\begin{proof}
There is a natural isomorphism $\Pic(W)\cong H^1(W,G_m)$. Using the Kummer sequence we get a map $c_1: \Pic(W) \to H^2(W,\Q_{\ell})$. 
Since $H^1(W,\Q_{\ell})\cong H^1(V,\Q_{\ell})$ (by Proposition~\ref{prpLef}) it follows from the Kummer sequence that the kernel of $c_1$ is contained in $\pi^*\Pic(V)$. From Proposition~\ref{prpLef} it follows that  $\iota$ acts trivialy on $H^2(W,\Q_{\ell})$ hence it acts trivially on  $\Pic(W)/\pi^*\Pic(V)$. Since $\iota$ acts also trivially on $\pi^*\Pic(V)$ it has to  act trivially on $\Pic(W)\otimes \Q$.

Consider now the map $c_1^*:CH^1(W)\to H^{2n-2}_c(W,\Q_{\ell})^* \to H^{2n-2}_c(W,\Q_{\ell})^{*-}$. The above discussion shows that $c_1^*(\Pic(W))$ is contained in the kernel of the second map. Since $N(W)$ is the kernel of the first map, we obtain that $\Pic(W)+N(W)=P(W)$ lies in the kernel of the composition. Hence we can factor the map $E(\overline{F_q}(V))\otimes \Q \to H^{2n-2}_c(W,\Q_{\ell})^{*-}$ as 
\[ E(\overline{F_q}(V))\otimes \Q \to CH^1(W)/P(W) \otimes \Q \to H^{2n-2}_c(W,\Q_{\ell})^{*-}.\]
From the previous lemma it follows that the composition is injective. In particular, the first map is injective.
\end{proof}

\begin{corollary} \label{corQCar}
Suppose $\rank \MW(\pi)\geq 1$. Then $W$ contains a Weil divisor that is not $\Q$-Cartier.
\end{corollary}

\begin{corollary} Suppose $W$ is smooth in codimension 3. Then $\rank E(\overline{\F_q}(V))=0$.
\end{corollary}

\begin{proof} From \cite[Th\'eor\`eme XI.3.13]{SGA2} it follows that a hypersurface that is smooth in codimension 3 is $\Q$-factorial, i.e., every Weil divisor is a $\Q$-Cartier divisor.
\end{proof}

We have the following direct consequence.
\begin{corollary}\label{corRank} Suppose $\dim V>2$ and that  $W$ is either smooth or has isolated singularities. Then the rank of $E(\F_q(V))$ equals 0.
\end{corollary}

\section{Densities}\label{sectDen}
In this section we work over a finite field $\F_q$.
 Let $k>0$ be an integer. Define
\[P_k:=\left\{ \begin{array}{cl}   
H^0(\cL^k) \times \{0\} \times H^0(\cL^{3k}) \times H^0(\cL^{4k}) \times H^0(\cL^{6k}) & \mbox{if } p=2,\\
\{0\} \times H^0(\cL^{2k}) \times \{0\} \times H^0(\cL^{4k}) \times H^0(\cL^{6k}) &\mbox{if }  p=3,\\
 \{0\} \times  \{0\} \times \{0\} \times H^0(\cL^{4k}) \times H^0(\cL^{6k}) &\mbox{if } p>3.
          \end{array}\right.
.\]
With a point
$\mathbf{a}:=(a_1,a_2,a_3,a_4,a_6)\in P_k$
we associate the hypersurface
\[ W_{\mathbf{a}} = V(-Y^2Z-a_1XYZ-a_3YZ^2+ X^3+a_2X^2Z+a_4XZ^2+a_6Z^3 ) \subset \Ps(\cE_k) \]
which has a fibration $\pi_{\mathbf{a}}:W_\mathbf{a} \to V$ in cubic curves with a marked point. 
Let 
\[ M_k=\left\{ \mathbf{a}\in P_k \mid\mbox{ there exists a } P\in V \mbox{ such that } \pi_{\mathbf{a}}^{-1}(P)  \mbox{ is smooth} \right\}.\]
For $\mathbf{a}\in M_k$ denote with $E_{\mathbf{a}}$ the corresponding elliptic curve over $\F_q(V)$.
Within $M_k$ we have the subsets
\[ U_k^{({i})}=\{ \mathbf{a} \in M_k \mid W_{\mathbf{a}} \mbox{ is smooth in codimension } i\}.
\]
So $U_k^{(n)}$ is the complement of the discriminant.
Define
\[ R_k^{(r)}= \{ \mathbf{a} \in M_k \mid  \rank E_{\mathbf{a}}(\F_q(V)) =r\} .\]
We know from Corollary~\ref{corQCar}
\[ U_k^{(3)} \subset R_k^{(0)}. \]
(Actually, in the previous section it was shown that all element of $U_k^{(3)}$ that correspond to minimal Weierstrass equations are contained in $R_k^{(0)}$. A hypersurfaces that is not a minimal Weierstrass equations is singular in codimension 2, hence is not contained in $U_k^{(3)}$. Later on we will show that the $\mathbf{a}$ defining hypersurfaces with non-isolated singularities have density zero, hence we might even completely disregard the non-minimal hypersurfaces.)

For a subset $S \subset \bigcup_i P_i$ define
\[ \mu(S) = \lim_{k\to \infty} \frac{\# S \bigcap\left( \bigcup_{i=1}^k P_i\right)}{\#   \bigcup_{i=1}^k\left( P_i \right)}\]
whenever this exists. With $\overline{\mu}(S)$ and $\underline{\mu}(S)$ we denote the same quantity where we replace $\lim$ by $\limsup$ or $\liminf$ respectively.

We will now prove that if $n\geq 4$ then $\mu(U_k^{(3)})=1$ and that $\mu(U_k^{(3)})=\zeta_V(3)^{-1}$ if $n=3$. This suffices to prove Theorem~\ref{mainThm}.

In this section we rely on several of the ideas presented by Poonen in \cite{Poo}. In the presentation we focus on the difference between our situation and that of Poonen.

\begin{lemma}\label{lemLowDeg}
 Let $V_{<r} \subset V$ be the subset of points of $V$ of degree  at most $r$. Define
\[ \cP_r:= \{ \mathbf{a} \in \cup_i P_i \mid W_{\mathbf{a}} \mbox{ is smooth at all points in } \pi_{\mathbf{a}}^{-1}(P) \mbox{ for all } P\in V_{<r} \}
\]
Then 
\[ \mu(\cP_{r}) = \prod_{P\in V_{<r} } (1-q^{-(m+1)\deg(P)}).\]
\end{lemma}

\begin{proof}
 Let $V_{<r}=\{R_1,\dots,R_s\}$. Let $t_1,\dots,t_m$ be local coordinates on $V$.

Suppose first that  $p>3$. Then $a_1=a_2=a_3=0$.
The hypersurface corresponding to a point $(0,0,0,a_4,a_6)\in P_k$ is singular at some point over $P_i$ of degree $e$ if and only if the $2\times (m+1)$ matrix
\[ \left( \begin{matrix}
       a_4(R_i) &    \frac{\partial a_4}{\partial t_1}(R_i) & \dots & \frac{\partial a_4}{\partial t_m}(R_i)\\ 

a_6(R_i)&           \frac{\partial a_6}{\partial t_1}(R_i) & \dots & \frac{\partial a_6}{\partial t_m}(R_i)\\
          \end{matrix}
\right)\]
is of the form
\[ \left( \begin{matrix}
       -3t^2 &  0 & \dots & 0 \\ 

2t^3 & 0 & \dots & 0 \\
          \end{matrix}
\right)
\mbox{ or } 
\left( \begin{matrix}
       -3\alpha^2 &  \mathbf{v} \\ 

2\alpha ^3 & \alpha \mathbf{v} \\
          \end{matrix}
\right).
\]
where $\alpha,t\in \F_{q^e}$ and $\mathbf{v} \in \F_{q^e}^m\setminus \{ 0 \}$.

Now, let $m_i$ be the ideal sheaf of $R_i$ and  $Y_i$ be the subscheme of $V$ corresponding to the ideal of $m_i^2$. Then the above conditions define a subset $S_i$ of the $2m+2$-dimensional $\F_{q^e}$-vector space $H^0(Y_i,\str_{Y_i}) \times  H^0(Y_i,\str_{Y_i})$ containing $q^{e(m+1)}$ elements.

Hence $(a_4,a_6)$ belong to  $\cP_r$ if and only if it is in the inverse image of 
\[\prod_{i=1}^s ((H^0(Y_i,\str_{Y_i}) \times H^0(Y_i,\str_{Y_i}))\setminus S_i)\]
under the map
\[ P_k \to H^0(Y,\cL^{2k}\otimes \str_Y) \times H^0(Y,\cL^{3k}\otimes \str_Y) \cong \prod_{i=1}^s ((H^0(Y_i,\str_{Y_i}) \times H^0(Y_i,\str_{Y_i}).\]
Since $\cL$ is ample it follows that for $k \gg 1$ the map
\[ H^0(V,\cL^{4k}) \times H^0(V,\cL^{6k}) \to H^0(Y,\cL^{2k}\otimes \str_Y) \times H^0(Y,\cL^{3k}\otimes \str_Y)\]
is surjective \cite[Lemma 2.1(a)]{Poo}. Hence
\[ \mu(\cP_r) = \prod_{i=1}^s \frac{ q^{2(m+1)\deg R_i}-q^{(m+1)\deg R_i}}{q^{2(m+1)\deg R_i}} = \prod_{i=1}^s (1-q^{-(m+1)\deg R_i}).\]

Suppose now that $p=2$. In this case  $a_1,a_3,a_4,a_6$ vary. There is a singular point of $W_{\mathbf{a}}$ in the fiber of $R$ if and only if
\begin{itemize}
 \item $a_1(R)\neq 0, x=a_3(R)/a_1(R), y=(3x^2+a_4(R))/a_1(R)$, for all $j=1,\dots m$ we have
\[ y^2+(a_1)_{t_j}(R)xy+(a_3)_{t_j}(R)y=x^3+(a_4)_{t_j}(R)x+(a_6)_{t_j}(R)\]
and
\[y^2+a_1(R)xy+a_3(R)y=x^3+a_4(R)x+a_6(R).\]
This defines a subset of 
$H^0(Y_i,\str_{Y_i})^4$ containing \[(q^{(m+1)\deg R}-q^{m\deg R})q^{2(m+1)\deg R} \] elements. (I.e., $a_2,a_3$ are free, $a_4$ does not vanish at $R_i$  and $a_6$ is completely determined by $a_2,a_3,a_4$.)
\item $a_1(R)=0, a_3(R)=0$, $3x^2=a_4(R)$ (this fixes $x$), $y^2=a_6(R)$ (this fixes $y$) and for all $j$ we have
\[ y^2+(a_1)_{t_j}(R)xy+(a_3)_{t_j}(R)y=x^3+(a_4)_{t_j}(R)x+(a_6)_{t_j}(R).\]
This defines a subset containing \[q^{2m\deg(R)} q^{(m+1)\deg R} q^{\deg(R)}=q^{(3m+2)\deg(R)}\] elements in $H^0(Y_i,\str_{Y_i})^4$. 
\end{itemize}
The union of these two sets consists of  $q^{3(m+1)\deg(R)}$ elements. Hence the fraction of good elements is
\[ 1-q^{-(m+1)\deg(R)},\]
as it was in the case $p>3$. The rest of the proof in this case is analogous to the case $p>3$.

If $p=3$ then  $a_2,a_4,a_6$ vary.
Assume first $a_2(R)\neq 0$ then in order to have a singularity of $W_{\mathbf{a}}$ in the fiber over $p$ we need to have $x=a_4(R)/a_2(R),y=0$. Now $a_6(R)$ and the value of its derivatives are completely determined by $a_2,a_4$ and their derivatives. 
 Hence we get a set of size
\[        (q^{(m+1)\deg(R)}-q^m) q^{(m+1)\deg(R)}.\]

If $a_2(R)=0$ then also $a_4(R)=0$, $a_6(R)$ is free and yields $x$, all partials of $a_6$ are determined by $x$ and the partials of $a_2,a_4$. Hence we get
\[ q^{2m\deg(R)} q^{\deg(R)}.  \]
bad elements.

In total we find
\[       (q^{(m+1)\deg(R)}-q^m) q^{(m+1)\deg(R)} +q^{2m\deg(R)} q^d=q^{2(m+1)\deg{R}}\]
bad elements. Hence the fraction of good elements is
\[ 1-\frac{q^{2(m+1)\deg(R)}}{q^{3(m+1)\deg(R)}} = 1-q^{(m+1)\deg(R)}. \]
We can finish this case analogous to the previous two cases.
\end{proof}

\begin{lemma}
 Let $R\in V$ be a closed point of degree $e$, with $e\leq k/(6m+6)$. Then the fraction of $\mathbf{a} \in  P_k$  such that the corresponding hypersurface $Y$ is singular at some point over $R$ equals $q^{-(m+1)e}$.
\end{lemma}
\begin{proof} Let $m_P$ be the maximal ideal of $P$, let $Y$ be the scheme associated with $m_P^2$, let $S\subset  P_k$ be the subset of bad elements constructed in the proof of the previous lemma.

Since $\cL$ is very ample we can use it to embed $V$ in projective space. From \cite[Lemma 2.1]{Poo} it follows that
\[ \psi_k :\cP_k \to\left\{ 
\begin{array}{cl} H^0(\str_Y(4k))\times H^0(\str_Y(6k)) & \mbox{if } p>3 \\
  H^0(\str_Y(2k))\times H^0(\str_Y(4k))\times H^0(\str_Y(6k)) & \mbox{if } p=3 \\
    H^0(\str_Y(k))\times H^0(\str_Y(3k))\times H^0(\str_Y(4k))\times H^0(\str_Y(6k)) & \mbox{if } p=2\end{array}\right.
\]
is surjective  if $6k>(m+1)e$.

Now, $\mathbf{a}\in M_k$ defines a  hypersurface singular at a point over $R$ if and only if it is mapped to a point in $S$ under $\psi_k$. In the previous lemma we proved that
\[ \frac{\# S}{\# H^0(Y,\str_Y)^g} = \frac{1}{q^{(m+1)e}}\]
From this the result follows.
\end{proof}

Analogously to  \cite[Lemma 2.4]{Poo} one obtains the following lemma:

\begin{lemma}
 Define 
\[ \cQ_r^{medium} = \left \{\mathbf{a} \in P_k \left| \begin{array}{l} \mbox{there exists a } P \in V \mbox{ with }\\ r\leq \deg P \leq \frac{k}{6(m+1)},\\ \mbox{ such that } Y \mbox{ is singular at a point over } P.\end{array}\right. \right\}\]
Then $\lim_{r\to \infty} \mu(\cQ_r^{medium}) =0$. 
\end{lemma}

\begin{proof}
 See \cite[Lemma 2.4]{Poo}.
\end{proof}

\begin{lemma}
 Define 
\[ \cQ^{high} = \left \{\mathbf{a} \in P_k \left| \begin{array}{l} \mbox{there exists a } P \in V \mbox{ with }\\ \deg P \geq \frac{k}{6(m+1)}\\ \mbox{ such that } Y \mbox{ is singular at a point over } P.\end{array}\right. \right\}\]
Then $\overline{\mu}(\cQ^{high}) =0$. 
\end{lemma}

\begin{proof}
We treat first the case $p>3$.

Fix a $a_4\in H^0(\cL^{4k})$.
Let $P\in V$ be a point, $t_1,\dots,t_m$ local coordinates.
If  the hypersurface corresponding to $(a_4,a_6)$ is singular over $P$ then 
 \begin{itemize}
 \item $x^2=\frac{1}{3} a_4(P)$,
 \item $(a_6)_{t_i}(P)=-x(a_4)_{t_i}(P)$ and
 \item $a_6(P)=-x^3-a_4(P)x$.
 \end{itemize}

 We show now that the density of $a_6$ satisfying the above three conditions is zero. In order to do this we partially decouple the derivatives as in \cite{Poo}. I.e., as in \cite[Lemma 2.6]{Poo} we may replace $V$ by an affine open, and we choose coordinates $t_1,\dots, t_m$ on $V$ as in \emph{loc. cit.} Write
 \[ a_6=f_0+\sum_i (h_i^p t_i) +h^p\]
 Then $D_i=Df_0+h_i^p$. Inside $V\times \A^1$ define $W_0$ as the vanishing set of $\{3x^2+a_4(P)\}$. Set $H_i=\{D_i a_6 = -xD_i a_4\}$. Let $W_i:=W_{i-1}\cap H_i$.  Similarly as in \cite[Lemma 2.6]{Poo} we get that, conditioned  on the choice of $f_0,h_1,\dots,h_i$ such that $\dim H_i=m-i$, the probability that 
$H_{i+1}$ has dimension $m-i-1$ is $1-o(1)$ and that, conditioned on the choice of $f_0,h_1,\dots,h_m$ such that $W_m$ is finite, the probability that $W_m\cap \{x^3+a_4(P)+a_6(P)\} \cap U_{\geq d/(m+1)}$ is empty is $1-o(1)$. The key assumption in this part of the proof is that $\deg(P)\geq k/6(m+1)$.
 
 From this we deduce that the probability that $y^2=x^3+a_4x+a_6$ is not smooth over a point of degree at least $d/(m+1)$ is bounded by $\prod_{i=0}^{m+1}(1-o(1))=1-o(1)$.

 In characteristic 2 and 3 can basically apply the same trick:
 
 In characteristic 2 we have additional polynomials $a_1,a_3$. We split our set up in the elliptic $n$-folds that have a singular point in a fiber over $P$ such that  $a_1(P)\neq 0$ and the set where such a singular point satisfies $a_1(P)=0$.
 
 In the former case the $x$ and $y$ are determined by $a_1(P),a_2(P)$ and $a_4(P)$ (see Lemma~\ref{lemLowDeg}). Hence given $a_1,a_2,a_4$ we look for $a_6\in H^0(\cL^{6k})$ such that
\begin{itemize}
\item $a_6(P)_{z_i}=y^2+a_1(P)_{z_i}xy+a_3(P)_{z_i} y-x^3-a_4(P)_{z_i} x$
\item $a_6(P)=y^2+a_1(P)xy+a_3(P)y-x^3-a_4(P)x$
\end{itemize}
We can decouple $a_6$ as before, showing that $P$ of degree at least $k/(6m+6)$ that satisfy the above system have density 0.

Now there might be singular points such that $a_1(P)=0$. This forces $a_3(P)=0$ and $3x^2-a_4(P)=0$ and $a_6(P)=y^2-x^3-a_4(P)$. So given $a_4,a_6$ the final two equations determine $x$ and $y$ uniquely (or there is no solution).

We first bound the set of $\mathbf{a}$ such that  a singularity in the fiber over $P$ with $a_1(P)=0$ and such that the $y$-coordinate does not vanish.
This means that  given $a_4,a_6$ we have a solution for $x$ and $y$. Since $y$ is non-zero we have:
\[ (a_3)_{z_i}(P) = \frac{(a_6)_{z_i}(P)xy-(a_4(P))_{z_i} x}{y} \]
In this case  we can decouple $a_3$ as above.
Now the final case, where $y=0, a_1(P)=0$ can be treated similarly.

In characteristic $3$ we have polynomials $a_2,a_4,a_6$. If $a_2(P)\neq 0$ then $2a_2(P)x+a_4(P)=0$, hence $x$ is determined. To study the density of the solutions of
\[ a_2(P)_{z_i}x^2+a_4(P)_{z_i}=-a_6(P)\]
and $x^3+a_2(P)x^2+a_4(P)x+a_6$ we can decouple the variables in $a_6$ and proceed as above.

In case that $a_2(P)=0$ then $a_4(P)=0$ and $x^3=-a_6(P)$. Hence $x$ is determined by $a_6$ and we can decouple with respect to $a_4$ provided that $x\neq 0$.

For the case that give rise to singular points with $x=0$ one notices that  $a_6$ is singular at $P$. This can be treated as in Poonen \cite[Lemma 2.6]{Poo}.
 \end{proof}

\begin{theorem} With notation as before
\[ \mu(U_k(3)) = \left\{\begin{array}{cl} 1 & \mbox{if } \dim V>2 \\ \zeta_V(3)^{-1} & \mbox{if } \dim V=2 \end{array}\right.\]
\end{theorem}
\begin{proof}
If $\dim V>2$ and $\mathbf{a}\in U_k^{(3)}$ then $W_{\mathbf{a}}$ has a singularity over a point $P$ of arbitrary high degree. From the previous lemma it follows that $\mu(U_k^{(3)})=0$.

If $\dim V=2$ the proof is the same as the proof of Theorem 1.2 in Poonen \cite[Section 2.4]{Poo}.
\end{proof}

\end{document}